\documentclass[11pt]{amsart}

\usepackage{amssymb,xcolor}
\definecolor{mred}{rgb}{0.7,0,0}
\definecolor{mblue}{RGB}{034,113,179}

\usepackage[ocgcolorlinks,colorlinks=true,citecolor=mblue,linkcolor=mred,urlcolor=grey,backref=page]{hyperref}
\usepackage[letterpaper,total={6in, 8.5in},top=1.3in,asymmetric]{geometry}

\usepackage[colorinlistoftodos]{todonotes}

\numberwithin{equation}{section}

\newtheorem{Theorem}{Theorem}[section]
\newtheorem{Lemma}[Theorem]{Lemma}
\newtheorem{Corollary}[Theorem]{Corollary}
\newtheorem{Proposition}[Theorem]{Proposition}

\theoremstyle{definition}
\newtheorem{Definition}[Theorem]{Definition}
\theoremstyle{remark}
\newtheorem{Remark}[Theorem]{Remark}

\newtheorem{Example}[Theorem]{Example}

\def \dim{{\mbox {dim}}\,}

\def\R{\mathbb{R}}

\def \la{\lambda}

\def \re{{\mathbb R}}

\def \C{{\mathbb C}}

\def \ov{\overline}

\def \0{\lambda_{0}}
\def \la{\lambda}
\def \ga{\gamma}

\def\d{d}

\begin{document}
\title[Differentials, Thermostats and Anosov flows]{Holomorphic Differentials, Thermostats and Anosov flows}
\author[T.~Mettler]{Thomas Mettler}
\address{Institut f\"ur Mathematik, Goethe-Universit\"at Frankfurt, 60325 Frankfurt am Main, Germany}
\email{mettler@math.uni-frankfurt.de}
\author[G.P.~Paternain]{Gabriel P.~Paternain}
\address{ Department of Pure Mathematics and Mathematical Statistics,
University of Cambridge,
Cambridge CB3 0WB, England}
\email {g.p.paternain@dpmms.cam.ac.uk}

\date{May 18, 2018}

\begin{abstract} We introduce a new family of thermostat flows on the unit tangent bundle of an oriented Riemannian $2$-manifold. Suitably reparametrised, these flows include the geodesic flow of  metrics of negative Gauss curvature and the geodesic flow induced by the Hilbert metric on the quotient surface of divisible convex sets. We show that the family of flows can be parametrised in terms of certain weighted holomorphic differentials and investigate their properties. In particular, we prove that they admit a dominated splitting and we identify special cases in which the flows are Anosov. In the latter case, we study when they admit an invariant measure in the Lebesgue class and the regularity of the weak foliations.
\end{abstract}

\maketitle

\section{Introduction}

We introduce a new family of flows on the unit tangent bundle $SM$ of a closed oriented Riemannian $2$-manifold $(M,g)$ of negative Euler characteristic. The flows are (generalised) thermostat flows and are generated by $C^{\infty}$ vector fields of the form $F:=X+(a-V\theta)V$, where $X,V$ denote the geodesic and vertical vector fields on $SM$, $\theta$ is a $1$-form on $M$ -- thought of as a real-valued function on $SM$ -- and $a$ represents a differential $A$ of degree $m\geqslant 2$ on $M$. The triple $(g,A,\theta)$ determining the flow is subject to the equations
\begin{equation}\label{eq:maineqs}
K_g=-1+\delta_g\theta+(m-1)|A|^2_g\quad \text{and}\quad \ov{\partial}A=\left(\frac{m-1}{2}\right)\left(\theta-i\star_g\theta\right)\otimes A,
\end{equation}
where $i=\sqrt{-1}$ and where $K_g$ denotes the Gauss--curvature, $\delta_g$ the co-differential and $\star_g$ the Hodge-star with respect to $g$ and the orientation. The case $m=3$ of these equations appeared previously in ~\cite{Met} (assuming $\theta$ is closed), where it is related to certain torsion-free connections on $TM$ which admit an interpretation as Lagrangian minimal surfaces. Here we prove that our flows admit a dominated splitting and moreover, that this family of flows admits a parametrisation in terms of holomorphic data. Indeed, we show that a triple $(g,A,\theta)$ satisfying the equations~\eqref{eq:maineqs} determines a holomorphic line bundle structure on the smooth complex line bundle $L_m:=\Lambda^2(TM)^{(m-1)/2}\otimes \C$, so that the ``weighted differential'' $P=\left(\det g\right)^{-(m-1)/4}\otimes A$ is a holomorphic section of $L_m\otimes K_M^m$ and such that a certain negative curvature condition holds. Here $K_M$ denotes the canonical bundle of $(M,g)$. Conversely, given a closed hyperbolic Riemann surface $(M,[g])$, a holomorphic line bundle structure on $L_m$ and a holomorphic section $P$ of $L_m\otimes K^m_{M}$ satisfying a certain negative curvature condition, we construct a triple $(g,A,\theta)$ solving~\eqref{eq:maineqs} and hence one of our flows, by using the uniformisation theorem and by solving an algebraic equation only. 

In~\cite{W2}, Wojtkowski introduced W-flows by suitably reparametrising the geodesics of a Weyl connection (or conformal connection). We show that the case where $A$ vanishes identically corresponds to W-flows associated to conformal connections on the tangent bundle of a surface that have negative definite symmetrised Ricci curvature. In particular, we recover~\cite[Theorem 5.2]{W2}, by showing that the flow associated to a triple $(g,0,\theta)$ solving~\eqref{eq:maineqs} is Anosov. This is achieved by providing sufficiency conditions for a general thermostat flow to admit a dominated splitting and to have the Anosov property, see Proposition~\ref{prop:basic} and Theorem~\ref{thm:workhorse}.   

We then turn to the case where $\theta$ vanishes identically, so that $A$ is holomorphic, hence we have
\begin{equation}\label{eq:maineqshol}
K_g=-1+(m-1)|A|^2_g \quad \text{and} \quad \ov{\partial} A=0.
\end{equation}
Note that applying standard quasi-linear elliptic PDE techniques we obtain a unique solution $g$ to~\eqref{eq:maineqshol} for every holomorphic differential $A$ on $(M,[g])$, see Remark~\ref{rmk:exsolpde}. The equations~\eqref{eq:maineqshol} admit an interpretation as~\textit{coupled vortex equations}, see in particular~\cite[\S 5]{DM}. The case $m=2$ was considered in~\cite{P07} in the context of Anosov thermostats admitting smooth weak bundles (see Section \ref{section:2-3} for more details). In the case $m=3$, the first equation is known as Wang's equation in the affine sphere literature. In~\cite{Wang}, Wang related its solution to a complete hyperbolic affine $2$-sphere in $\R^3$, in particular $g$  is known as the~\textit{Blaschke metric}. Moreover, for $m=3$, a pair $(g,A)$ on $M$ solving~\eqref{eq:maineqshol} defines a ~\textit{properly convex projective structure} on $M$ and hence turns $M$ into a properly convex projective surface, see~\cite{Lab} and~\cite{Loftin}. The universal cover $\Omega$ of a properly convex projective surface of negative Euler characteristic is a strictly and properly convex domain in the projective plane $\mathbb{RP}^2$ which admits a cocompact action by a group $\Gamma$ of projective transformations. Consequently, we obtain a (two-dimensional)~\textit{divisible convex set}. Since $\Omega$ is convex, it is equipped with the Hilbert metric and moreover, the Hilbert metric descends to define a Finsler metric on the quotient surface $M\simeq \Omega/\Gamma$, see in particular~\cite{KimPap} for a nice survey of these ideas. We observe that the geodesic flow of the Finsler metric is a $C^1$ reparametrisation of the flow we associate to the pair $(g,A)$. Benoist has shown~\cite{BConvDivI} that if $(\Omega,\Gamma)$ is a divisible convex set (not necessarily two-dimensional), then the geodesic flow of the Finsler metric $F$ induced on $\Omega/\Gamma$ -- henceforth just called the Hilbert geodesic flow -- is Anosov if and only if $\Omega$ is strictly convex. Since the Anosov property is invariant under reparametrisation, we may ask if the thermostat flow associated to a pair $(g,A)$ solving~\eqref{eq:maineqshol} is Anosov for all $m\geqslant 2$. This is indeed the case, we obtain:
\setcounter{section}{5}
\begin{Theorem} Let $(g,A)$ be a pair satisfying the coupled vortex equations $\bar{\partial}A=0$ and
$K_g=-1+(m-1)|A|^{2}_{g}$. Then the associated thermostat flow is Anosov.
\end{Theorem}
\setcounter{Theorem}{4}  

The hyperbolicity properties of thermostats satisfying \eqref{eq:maineqs} are not apparent. To expose them,
we first conjugate the derivative cocycle to another one in which we can see the effect of equations \eqref{eq:maineqs}.
This conjugation requires a careful choice of gauge, but once that is established, standard methods using quadratic forms give rise to a dominated splitting. To upgrade this dominated splitting to hyperbolic as in the case of Theorem \ref{thm:upgradeanosov} requires an additional ingredient in the form of Lemma \ref{lemma:analogue_BH} below which asserts that $K_{g}<0$;  this gives control on the potentially problematic size of $A$.  

In the same way as geodesic flows are paradigms of conservative systems, thermostats may be seen as paradigms of dissipative systems. The special case of Gaussian thermostats ($a=0$) has provided interesting models in nonequilibrium statistical mechanics \cite{Ga,GaRu,Ru1}.
The next theorem shows that Anosov thermostat flows determined by the coupled vortex equations are indeed dissipative except when $A=0$.

\begin{Theorem}Let $(g,A)$ be a pair satisfying the coupled vortex equations $\bar{\partial}A=0$ and
$K_g=-1+(m-1)|A|^{2}_{g}$.  Then the associated thermostat flow preserves an absolutely continuous measure if and only if $A$ vanishes identically.
\label{thm:entp}
\end{Theorem}

We remark that due to a theorem of Ghys \cite{Ghy0} Anosov thermostat flows are H\"older orbit equivalent to the geodesic flow of (any) negatively curved metric of $M$ and hence transitive (to be precise, \cite{Ghy0} establishes a topological equivalence and the  H\"older orbit equivalence follows from \cite[Theorem 19.1.5]{KH}).

In~\cite{BConvDivI}, Benoist also observes that the regularity of the weak foliations of the Hilbert geodesic flow coincides with the regularity of the boundary of the divisible convex set $(\Omega,\Gamma)$. By a result of Benz\'ecri~\cite{Benz}, the boundary has regularity $C^2$ if and only if $\Omega$ is an ellipsoid, in which case the induced Finsler metric is Riemannian and hyperbolic. Hence one might speculate that if a solution to the coupled vortex equations~\eqref{eq:maineqshol} gives rise to an Anosov flow having a weak foliation of regularity $C^2$, then $A$ vanishes identically. While we cannot prove this in general, we use Theorem \ref{thm:entp} to resolve the odd case:
\setcounter{section}{7}
\setcounter{Theorem}{0}
\begin{Theorem}Suppose an Anosov thermostat given by the coupled vortex equations has a weak foliation of class $C^{2}$ and $m$ is odd. Then $A$ vanishes identically.
\end{Theorem}

The orbits of our flow -- when projected to the surface $M$ -- define what is known as a path geometry on $M$, that is, a prescription of a path on $M$ for every direction in each tangent space. In the case where $A$ vanishes identically the paths are the geodesics of a hyperbolic metric and in the case where $m=3$ the paths are the geodesics of a properly convex projective structure. In both cases, the path geometry is~\textit{flat}, by which we mean it is locally equivalent to the path geometry of great circles on the $2$-sphere. In the final section of the article we show:
\setcounter{section}{8}
\setcounter{Theorem}{2}
\begin{Theorem}
Let $(g,A)$ be a pair satisfying the coupled vortex equations $\bar{\partial}A=0$ and $K_g=-1+(m-1)|A|^{2}_{g}$. Then the path geometry defined by the thermostat associated to $(g,A)$ is flat if and only if $m=3$ or $A$ vanishes identically.
\end{Theorem} 

Holomorphic differentials appear naturally in higher Teichm\"uller theory and here we briefly provide some context for our results while referring the reader to the recent survey~\cite{Wienhard} by Wienhard for a nice introduction to this currently very active research topic. Generalizing Teichm\"uller space, Hitchin~\cite{Hitchin} identified a connected component $\mathcal{H}(M,G)$ -- nowadays called the~\textit{Hitchin component} -- in the representation variety $\mathrm{Hom}\left(\pi_1M,G\right)/G$, where $M$ is a connected closed oriented surface of negative Euler characteristic and $G$ a real split Lie group. Fixing a conformal structure $[g]$ on $M$, Hitchin used the theory of Higgs bundles~\cite{Hitchin2} to provide a parametrisation of $\mathcal{H}(M,G)$ in terms of holomorphic differentials on $(M,[g])$. While Hitchin's parametrisation of $\mathcal{H}(M,G)$ relies on the choice of an arbitrary conformal structure $[g]$ on $M$, Labourie~\cite{Lab2} was recently able to construct a canonical parametrisation of $\mathcal{H}(M,G)$ in the case where $G$ is $\mathrm{PSL}(3,\R)$, $\mathrm{PSp}(4,\R)$ or the split form $\mathrm{G}_{2,0}$ of the exceptional group $\mathrm{G}_2$ (see also~\cite{Lab} and~\cite{Loftin} for the case $G=\mathrm{PSL}(3,\R)$). More precisely, Labourie obtains a mapping class group equivariant identification of $\mathcal{H}(M,G)$ with the fibre bundle over Teichm\"uller space whose fibre at $J$ is $H^0(M,K_{M,J}^3)$, $H^0(M,K_{M,J}^4)$ and $H^0(M,K_{M,J}^6)$ respectively. By the work of Goldman~\cite{Goldman} and Choi--Goldman~\cite{ChoiGoldman} the component $\mathcal{H}(M,\mathrm{PSL}(3,\R))$ consists of (conjugacy classes of) monodromy representations of properly convex projective structures on $M$ and this together with the work of Labourie~\cite{Lab,Lab2} and Loftin~\cite{Loftin} yields the aforementioned description of properly convex projective structures in terms of pairs $([g],A)$ with $A$ a holomorphic cubic differential.  

Using the equivariant flag curve of Labourie~\cite{Lab3}, Potrie--Sambarino~\cite{PotrieSambarino} associate several Anosov flows to every representation $\rho$ in a certain neighbourhood of the Fuchsian locus in $\mathcal{H}(M,\mathrm{PSL}(n,\R))$, $n\geqslant 4$. In particular, using the canonical embeddings
$$
\mathcal{H}(M,\mathrm{PSp}(4,\R))\subset \mathcal{H}(M,\mathrm{PSL}(4,\R)) \quad \text{and} \quad \mathcal{H}(M,\mathrm{G}_{2,0})\subset \mathcal{H}(M,\mathrm{PSL}(7,\R)),
$$
the work of Labourie~\cite{Lab3,Lab2} and Potrie--Sambarino~\cite{PotrieSambarino} yields examples of Anosov flows for certain quartic and sixtic holomorphic differential on $(M,[g])$. It would be interesting to know how these flows relate to the flows introduced here. We plan to investigate this in future work. 

\subsection*{Acknowledgements} 
The authors are grateful to Nigel Hitchin, Rafael Potrie and Andy Sanders for helpful conversations and the anonymous referee for her/his careful reading and many useful suggestions. GPP was partially funded by EPSRC grant EP/M023842/1. 

\setcounter{section}{1}

\section{Preliminaries on general thermostats}

Let $M$ be a closed oriented surface equipped with a Riemannian metric $g$, $SM$ its unit circle bundle
and $\pi:SM\to M$ the canonical projection. The latter is in fact
a principal $\mathrm{SO}(2)$-bundle and we let $V$ be the infinitesimal
generator of the action of $\mathrm{SO}(2)$.

Given a unit vector $v\in T_{x}M$, we will denote by $Jv$ the
unique unit vector orthogonal to $v$ such that $\{v,Jv\}$ is an
oriented basis of $T_{x}M$. There are two semibasic 1-forms $\omega_1$
and $\omega_2$ on $SM$, which are defined by the formulas:
\[(\omega_1)_{(x,v)}(\xi):=g\left(d_{(x,v)}\pi(\xi),v\right);\]
\[(\omega_2)_{(x,v)}(\xi):=g\left(d_{(x,v)}\pi(\xi),Jv\right).\]
The form $\omega_1$ is the canonical contact form of $SM$ whose Reeb vector
field is the geodesic vector field $X$.

A basic theorem in 2-dimensional Riemannian geometry asserts that
there exists a~unique 1-form $\psi$ on $SM$ -- the Levi-Civita connection form of $g$ -- 
such that $\psi(V)=1$ and
\begin{align}
d\omega_1&=-\omega_2\wedge\psi\label{riem1},\\
d\omega_2&=-\psi\wedge\omega_1\label{riem2},\\ 
d\psi&=-(K_g\circ\pi)\,\omega_1\wedge\omega_2,\label{riem3}
\end{align}
where $K_g$ denotes the Gaussian curvature of $g$. In fact, the form
$\psi$ is given by
\[\psi_{(x,v)}(\xi)=g\left( \frac{DZ}{dt}(0),Jv\right),\]
where $Z:(-\varepsilon,\varepsilon)\to SM$ is any curve with
$Z(0)=(x,v)$, $\dot{Z}(0)=\xi$ and $\frac{DZ}{dt}$ is the
covariant derivative of $Z$ along the curve $\pi\circ Z$.

For later use it is convenient to introduce the vector field $H$
uniquely defined by the conditions $\omega_2(H)=1$ and
$\omega_1(H)=\psi(H)=0$. The vector fields $X,H,V$ are dual to
$\omega_1,\omega_2,\psi$ and as a consequence of (\ref{riem1}--\ref{riem3}) they satisfy the commutation relations
\begin{equation}\label{comm}
[V,X]=H,\quad [V,H]=-X,\quad [X,H]=K_gV.
\end{equation}
Equations (\ref{riem1}--\ref{riem3}) also imply that the vector fields
$X,H$ and $V$ preserve the volume form $\omega_1\wedge d\omega_1$ and hence
the Liouville measure.
Note that the flow of $H$ is given by $R^{-1}\circ \phi^0_t\circ R$, where $R(x,v)=(x,Jv)$
and $\phi^0_t$ is the geodesic flow of $g$. 

Let $\lambda$ be an arbitrary smooth function on $SM$. For several of the results that we will describe below, we will not
need $\la$ to be a special polynomial in the velocities.
We consider a (generalised)~\textit{thermostat flow} on $(M,g)$, that is, a flow $\phi$ defined by 
\begin{equation}
\frac{D\dot{\gamma}}{dt}=\la(\gamma,\dot{\ga})\,J\dot{\ga}.
\label{eqgt}
\end{equation}
It is easy to check that
$$
F:=X+\lambda V
$$
is the generating vector field of $\phi$.

Now let $\Theta:=-\omega_1\wedge d\omega_1=\omega_1\wedge\omega_2\wedge\psi$. This volume form generates the
Liouville measure $d\mu$ of $SM$.

\begin{Lemma} We have:
\begin{align}
L_{F}\Theta &=V(\la)\Theta;\label{lie1}\\
L_{H}\Theta &=0;\label{lie2}\\
L_{V}\Theta &=0.\label{lie3}
\end{align}
\label{lied}
\end{Lemma}
\begin{proof} Note that for any vector field $Y$, $L_{Y}\Theta=d(i_{Y}\Theta)$, by Cartan's formula. Since
$i_{V}\Theta=\omega_1\wedge\omega_2=\pi^*\Omega_{a}$, where $\Omega_{a}$ is the area
form of $M$, we see that $L_{V}\Theta=0$.
Similarly, $L_{X}\Theta=L_{H}\Theta=0$.
Finally $L_{F}\Theta=L_{X}\Theta+L_{\la V}\Theta=d(i_{\la V}\Theta)=
V(\la)\Theta$.
\end{proof}

\subsection{Jacobi equations} It is easy to derive the ODEs governing the behaviour of $d\phi_{t}$ using the bracket relations
above. Given $\xi\in T_{(x,v)}SM$ (the initial conditions), if we write
\[d\phi_{t}(\xi)=xF+yH+uV\]
then 
\begin{align}
\dot{x} &=\lambda\,y;\label{jac1}\\
\dot{y} &=u;\label{jac2}\\
\dot{u} &=V(\lambda)\dot{y}-\kappa y,\label{jac3}
\end{align}
where $\kappa:=K_g-H\lambda+\lambda^2$.

\subsection{Quotient cocycle} We consider the rank two quotient vector bundle $E=TSM/\R F$. We use the notation $[\xi]$ with $\xi \in TSM$ for the elements of $E$. Note that $d\phi_t$ descends to the quotient to define a mapping  
$$
\rho : E\times \R \to E, \quad ([\xi],t)\mapsto \rho([\xi],t)=[d\phi_t(\xi)] 
$$
satisfying $\rho_{t}\circ\rho_{s}=\rho_{t+s}$ for all $t,s \in \R$. The basis of vector fields $(F,H,V)$ on $SM$ defines a vector bundle isomorphism $TSM \simeq SM\times\R^3$ and consequently an identification $E\simeq SM \times \R^2$. Therefore, for each $t \in \R$, we obtain a unique map $\Psi_t : SM \to GL(2,\R)$ defined by the rule
$$
\rho_t((x,v),w)=\left(\phi_t(x,v),\Psi_t(x,v)w\right)
$$
for all $((x,v),w) \in E\simeq SM\times \R^2$. The map $\Psi : SM \times \R \to GL(2,\R)$ satisfies
$$
\Psi_{t+s}(x,v)=\Psi_{s}(\phi_{t}(x,v))\Psi_{t}(x,v)
$$
for all $(x,v) \in SM$ and $t,s \in \R$, and hence defines an $GL(2,\R)$-valued cocycle on $SM$ with respect to the $\R$-action defined by $\phi$. Explicitly,  $\Psi_t$ is the matrix whose action on $\R^2$ is given by 
$$
\Psi_t(x,v): \left( \begin{array}{c} y(0) \\ \dot{y}(0) \end{array} \right) \mapsto \left( \begin{array}{c} y(t) \\ \dot{y}(t) \end{array} \right)
$$
where $\ddot{y}(t)-V(\lambda)(\phi_{t}(x,v))\dot{y}(t) + \kappa(\phi_t(x,v)) y(t) = 0$.

Note that for thermostats the 2-plane bundle spanned by $H$ and $V$ is in general {\it not} invariant under $d\phi_{t}$.

\subsection{Infinitesimal generators and conjugate cocycles}\label{subsection:cocycle} Given a cocycle $\Psi_{t}:SM\times\R\to GL(2,\R)$ we define its infinitesimal generator $\mathbb{B}:SM\to \mathfrak{gl}(2,\R)$ as
\[\mathbb{B}(x,v):=-\left.\frac{d}{dt}\right|_{t=0}\Psi_{t}(x,v).\]
The cocycle $\Psi_{t}$ can be recovered from $\mathbb{B}$ as the unique solution to
\[\frac{d}{dt}\Psi_{t}(x,v)+\mathbb{B}(\phi_{t}(x,v))\Psi_{t}(x,v)=0,\;\;\;\Psi_{0}(x,v)=\mbox{\rm Id}.\]
For the case of thermostats, it is immediate to check that
$$
\mathbb{B}=\begin{pmatrix} 0 & -1 \\ \kappa & -V\lambda\end{pmatrix}.
$$
Given a smooth map $\mathcal P:SM\to GL(2,\R)$ (a gauge) we can define a new cocycle by conjugation as
\[\tilde{\Psi}_{t}(x,v)=\mathcal{P}^{-1}(\phi_{t}(x,v))\Psi_{t}(x,v)\mathcal{P}(x,v).\]
It is easy to check that the infinitesimal generator $\tilde{\mathbb{B}}$ of $\tilde{\Psi}_{t}$ is related to $\mathbb{B}$ by
\begin{equation}
\tilde{\mathbb{B}}=\mathcal{P}^{-1}\mathbb{B}\mathcal{P}+\mathcal{P}^{-1}F\mathcal{P}.
\label{eq:conjugate}
\end{equation}

\section{Dominated splitting and hyperbolicity for thermostats}

We are interested in the questions: when is this cocycle hyperbolic? When does it have a dominated splitting? We start with some definitions.

\begin{Definition} The cocycle $\Psi_t$ is~\textit{free of conjugate points} if any non-trivial
solution of the Jacobi equation $\ddot{y}-V(\lambda)\dot{y}+\kappa y=0$ with $y(0)=0$ vanishes only at $t=0$.
\end{Definition}

\begin{Definition} The cocycle $\Psi_t$ is said to be~\textit{hyperbolic} if there
exists a splitting
$E=E^{u}\oplus E^{s}$ where $E^u,E^s$ are continuous $\rho$-invariant line subbundles of $TSM$, and constants $C>0$ and $0<\zeta<1<\eta$ such that 
for all $t>0$ we have
$$
\Vert\Psi_{-t}|_{E^u}\Vert\leqslant C\,\eta^{-t}\quad \text{and}\quad \Vert\Psi_{t}|_{E^s}\Vert\leqslant C\,\zeta^{t}.
$$ 
\end{Definition}
We also say:
\begin{Definition} The cocycle $\Psi_t$ is said to have a~\textit{dominated splitting} if there
is a continuous $\rho$-invariant splitting
$E=E^{u}\oplus E^{s}$, and constants $C>0$ and $0<\tau<1$ such that 
for all $t>0$ we have
\[\|\Psi_{t}|_{E^{s}(x,v)}\|\|\Psi_{-t}|_{E^{u}(\phi_{t}(x,v))}\|\leq C\,\tau^{t}.\]
\end{Definition}

Obviously hyperbolicity implies dominated splitting. It also implies that there are no conjugate points \cite{DP2}. Moreover the cocycle $\Psi_{t}$ is hyperbolic if and only if the thermostat flow $\phi$ is Anosov (cf. for instance \cite[Proposition 5.1] {W1} where it is proved that the subbundles $E^{s,u}$ of $E$ lift to subbundles of $TSM$ to give the usual definition of Anosov flow). We shall say that $\phi$ has a dominated splitting if $\Psi_{t}$ has a dominated splitting (this is the adequate notion of dominated splittings for flows, see e.g. \cite[Definition 1]{AR-H}). For the case of flows on 3-manifolds, as it is our case, the existence of a dominated splitting can produce hyperbolicity if one has additional information on the closed orbits. Indeed \cite[Theorem B]{AR-H} implies that if all closed orbits of $\phi$ are hyperbolic saddles, then $SM=\Lambda\cup\mathcal T$ where $\Lambda$ is a hyperbolic invariant set and $\mathcal T$ consists of finitely many normally hyperbolic irrational tori.

A very convenient way to establish the aforementioned properties for cocycles is to use quadratic forms as in \cite{L,W2, W3}. In particular, we have~\cite[Proposition 4.1 $\&$ Theorem 4.4]{W3}:
\begin{Proposition}[Wojtkowski]\label{ppn:sufsplit}
Let $Q$ be a continuous non-degenerate quadratic form on $E$. Suppose furthermore that the  derivative
$$
\dot{Q}\!\left([\xi]\right):=\left.\frac{\d}{\d t}\right\vert_{t=0}Q\!\left([d\phi_t(\xi)]\right)
$$
exists for all $[\xi] \in E$. Then $\Psi_t$ has a dominated splitting if $\dot{Q}([\xi])>0$ for all $[\xi]\neq 0$ with $Q([\xi])=0$. If the stronger property $\dot{Q}([\xi])>0$ for all $[\xi]\neq 0$ holds, then $\Psi_t$ is hyperbolic.
\end{Proposition}

In what follows it will be helpful to understand how the spaces $E^{u,s}$ are constructed using $Q$. This is explained in detail in \cite[Proposition 4.1]{W3}, so here we just give a brief summary adapted to our situation. We let $\mathcal L_{+}(x,v)$ denote the set of all 1-dimensional subspaces $W$ such that $Q_{(x,v)}$ is positive on $W$. The condition on the quadratic form $Q$ ensures that $\Psi_{t}$ acts as a contraction on $\mathcal L_{+}$ and hence there is a unique point of intersection
\begin{equation}
E^{u}(x,v)=\bigcap_{t>0}\Psi_{t}(\phi_{-t}(x,v))\mathcal L_{+}(\phi_{-t}(x,v)).
\label{eq:construction}
\end{equation}
All our quadratic forms $Q$ below will have the property that $Q(0,b)=0$ (using the identification $E\simeq SM \times \R^2$) and hence we can construct $E^{u}$ (and $E^s$) simply by applying the procedure \eqref{eq:construction} to the vertical subspace $\R(0,1)$, that is,
\begin{equation}
E^{u}(x,v)=\lim_{t\to\infty}\Psi_{t}(\phi_{-t}(x,v))\R\left( \begin{array}{c} 0 \\ 1 \end{array} \right).
\label{eq:construction2}
\end{equation}

Let us put these ideas to use. Define ${\mathbb K}=\kappa +FV\lambda$.

\begin{Proposition} Assume $\mathbb{K}<0$. Then $\phi$ is Anosov.
\label{prop:basic}
\end{Proposition}

\begin{proof}We let $(a,b)$ denote the standard coordinates on $\R^2$. Using the identification $E\simeq SM\times \R^2$ we define a quadratic form on $E$ by the rule
\[Q_{(x,v)}(a,b)=(b-V(\lambda)a)a.\]
Then
\[Q_{\phi_{t}(x,v)}(\Psi_{t}(a,b))=(\dot{y}-V(\lambda)y)y,\]
where $y$ is the unique solution of 
$$\ddot{y}-V(\lambda)\dot{y}+\kappa y=0,
$$ with $y(0)=a$ and $\dot{y}(0)=b$. 
A simple calculation shows that
\[\dot{Q}=\frac{d}{dt}Q_{\phi_{t}(x,v)}(\Psi_{t}(a,b))=-\mathbb{K}y^2+(\dot{y}-V(\lambda)y)\dot{y}.\]
Since $\mathbb{K}<0$ we see that 
$$
\left.\frac{d}{dt}\right|_{t=0}Q_{\phi_{t}(x,v)}(\Psi_{t}(a,b))>0
$$ for $(a,b)\neq 0$ and such that
$Q_{(x,v)}(a,b)=0$. Then Proposition~\ref{ppn:sufsplit} immediately implies that $\Psi_{t}$ has a dominated splitting. We can upgrade that to hyperbolic as follows.
If we let $z:=\dot{y}-V(\lambda)y$, then the quadratic form is just $zy$. By the construction of the subspaces
$E^{s,u}$ (cf. \eqref{eq:construction}) we see that $E^{s,u}$ do not contain neither $z=0$, nor $y=0$. Hence there exist continuous
functions $r^{s,u}:SM\to \mathbb{R}$ such that $H+r^{s,u} V\in E^{s,u}$. Moreover, we see that $r^{u}-V\lambda>0$ and $r^{s}-V\lambda<0$.
Consider now a solution with initial conditions $(y(0),\dot{y}(0))\in E^{u}$. Then $z=(r^{u}-V\lambda)y$ and
$\dot{z}=-\mathbb{K}y=-\mathbb{K}(r^{u}-V\lambda)^{-1}z$. This gives exponential growth for $z$ and hence the desired exponential
growth for $\Psi_{t}$ on $E^{u}$. Arguing in a similar way with $E^{s}$, we deduce that $\Psi_{t}$ is hyperbolic.
\end{proof}

\begin{Remark}{\rm By considering the quadratic form $Q=y\dot{y}$ we can deduce with a similar proof that if $\kappa<0$
 the thermostat flow $\phi$ is Anosov. This is because $\dot{Q}=\dot{y}^{2}-\kappa y^{2}+V(\lambda)y\dot{y}$.
 We have $r^{u}>0$ and hyperbolicity follows from $\dot{y}=r^{u}y$ when  $(y(0),\dot{y}(0))\in E^{u}$. }
\end{Remark}

In fact we can generalise this further as follows.

\begin{Theorem} Let $p:SM\to\mathbb{R}$ be a smooth function such that 
$$\kappa_{p}:=\kappa+Fp+p(p-V\lambda)<0.$$ Then
 $\phi$ has a dominated splitting. If in addition $\kappa_{p}+\frac{(V\lambda)^{2}}{4}<0$, then the flow is Anosov.
 \label{thm:workhorse}
\end{Theorem}

\begin{proof} The quadratic form to consider is $Q=zy$, where $z:=\dot{y}-py$. A calculation shows
 that
 \[\dot{Q}=z^{2}-\kappa_{p}y^{2}+zyV\lambda.\] 
 We see that $\dot{Q}>0$ whenever $zy=0$, but $(y,z)\neq 0$. The claim in the theorem again follows from Proposition~\ref{ppn:sufsplit}.
 Also note that
 \[\dot{Q}=\left(z-\frac{yV\lambda}{2}\right)^{2}-\left(\kappa_{p}+\frac{(V\lambda)^{2}}{4}\right)y^{2}>0,\]
 unless $(z,y)=0$. Hence the flow is Anosov by Proposition \ref{ppn:sufsplit}.
\end{proof}

\begin{Remark}\label{rem:upgrade} Let us see the main issue with upgrading the last theorem to ``hyperbolic'' as in the proof of Proposition \ref{prop:basic}.
Certainly we get continuous (H\"older in fact) functions $r^{s,u}$. To be definite consider the case of $E^u$ and initial conditions
$(y(0),\dot{y}(0))\in E^{u}$. Then $\dot{y}=r^{u}y$ and $z=(r^{u}-p)y$ with $r^{u}-p>0$ as before.
But now $\dot{z}=(V\lambda-p)z-\kappa_{p}y=(V\lambda-p-\frac{\kappa_{p}}{r^{u}-p})z$.
To get exponential growth we either need:
\begin{equation}
r^{u}>0,\;\;\text{or}\;\;\;V\lambda-p-\frac{\kappa_{p}}{r^{u}-p}>0
\label{eq:alternative}
\end{equation}
and it is not clear how to get any of these conditions in this generality. In the special cases above $p=0$ or $p=V\lambda$,
we do get one of these conditions.
In all these cases the function $r=r^{u,s}$ satisfies the Riccati equation
\[Fr+r^{2}-rV\lambda+\kappa=0,\]
which is easily derived using the invariance of $E^{s,u}$ and the Jacobi equation  $\ddot{y}-V(\lambda)\dot{y}+\kappa y=0$.
Observe that $h:=r-p$ satisfies the Riccati equation
\begin{equation}
Fh+h^{2}+h(2p-V\lambda)+\kappa_{p}=0.
\label{eq:ric2}
\end{equation}
Using \eqref{eq:construction2} we can also give a construction of functions $r^{u,s}$ at the level of the Riccati equation as follows. Fix $(x,v)$ and consider for each $R>0$, the unique solution $u_{R}$ to the Riccati equation
along $\phi_{t}(x,v)$
\[\dot{u}+u^2-uV\lambda+\kappa=0\]
satisfying $u_{R}(-R)=\infty$. Then \eqref{eq:construction2}  translates easily into
\begin{equation}
r^{u}(x,v)=\lim_{R\to \infty}u_{R}(0).
\label{eq:hopflimit}
\end{equation}
Note that $r^{u}(\phi_{t}(x,v))=\lim_{R\to\infty}u_{R}(t)$. These limiting solutions exist whenever the cocycle $\Psi_{t}$ has no conjugate points \cite{AD}. It is easy to check that in all the cases we consider below, the cocycle $\Psi_t$ is free of conjugate points.
\end{Remark}

\begin{Remark}{\rm This remark attempts to clarify the role of the function $p$ in terms of conjugate cocycles and infinitesimal generators as in Subsection \ref{subsection:cocycle}.  As we have already pointed out, the infinitesimal generator $\mathbb B$ for a thermostat is given by
$$
\mathbb B=\begin{pmatrix} 0 & -1 \\ \kappa & -V\lambda\end{pmatrix}.
$$
Consider a gauge transformation $\mathcal P:SM\to GL(2,\R)$ given by
\[\mathcal P=\begin{pmatrix} 1 & 0 \\ p & 1\end{pmatrix}.\]
A calculation using \eqref{eq:conjugate} shows that the conjugate cocyle $\tilde{\Psi}_{t}$ via $\mathcal{P}$ has infinitesimal generator given by
\[\tilde{\mathbb{B}}=\begin{pmatrix} -p & -1 \\ \kappa_{p} & -V\lambda+p\end{pmatrix}.\]
The cocycles $\Psi_{t}$ and $\tilde{\Psi}_{t}$ share the same dominated splitting/hyperbolicity properties by virtue of being conjugate, but the form of $\tilde{\mathbb B}$ exposes clearly the origins of these properties via $\kappa_{p}<0$ (cf. \cite[Introduction]{W3}). The trace of both matrices, which is $-V\lambda$ (minus divergence of $F$), indicates the dissipative nature of thermostats.

}
\end{Remark}

\section{Applications}

We consider now some special choices of $\lambda$. To this end let $\theta$ be a $1$-form on $M$ which we may equivalently think of as a function $\theta : SM \to \R$ satisfying $VV\theta=-\theta$. For later use we record that the co-differential of $\theta$ and its Hodge-star satisfy
\begin{equation}\label{eq:idcodiff}
\pi^*\delta_g\theta=-(X\theta+HV\theta), \qquad \pi^*(\star_g\theta)=-V(\theta)\omega_1+\theta\omega_2.
\end{equation}
Moreover, let $A$ be a differential of degree $m$ on $M$ with $m\geqslant 2$. By this we mean a section of the $m$-th tensorial power of the canonical bundle $K_{M}$ of $(M,g)$. Likewise, we may equivalently think of a differential $A$ of degree $m$ on $M$ as a real-valued function $a : SM \to \R$ satisfying $VVa=-m^2a$, explicitly, we obtain 
$$
\pi^*A=\left(Va/m+i a\right)\left(\omega_1+i \omega_2\right)^m,
$$
so that 
\begin{equation}\label{eq:normdiff}
\pi^*|A|^2_g=(Va)^2/m^2+a^2
\end{equation} 

The thermostat flows we investigate are of the form $\lambda=a-V\theta$. We will see next that they admit a dominated splitting provided a natural pair of equations is satisfied by the triple $(g,A,\theta)$. In order to derive these equations we first need a Lemma. 
\begin{Lemma}\label{lem:transnotthomasgabriel}
We have
\begin{equation}\label{eq:weakhol1}
\ov{\partial}A=\left(\frac{m-1}{2}\right)\left(\theta-i\star_g\theta\right)\otimes A\quad
\end{equation}
iff
\begin{align}
\label{eq:realcond}0&=XVa-mHa-(m-1)(\theta Va-maV\theta).
\end{align}
\end{Lemma}
\begin{Remark}\label{rmk:imagcond}
Note that applying $V$ we see that~\eqref{eq:realcond} is equivalent to
\begin{equation}\label{eq:imagcond}
0=(1-m)\left(HVa+mXa-(m-1)\left(m\theta a+V(\theta)V(a)\right)\right).
\end{equation}
\end{Remark}
\begin{proof}[Proof of Lemma~\ref{lem:transnotthomasgabriel}]
We use the complex notation $\tilde{a}=Va/m+i a$ and $\omega=\omega_1+i\omega_2$. Since $VVa=-m^2a$, we compute that there exist unique complex-valued functions $\tilde{a}^{\prime}$ and $\tilde{a}^{\prime\prime}$ so that
$$
\d \tilde{a}=\tilde{a}^{\prime}\omega+\tilde{a}^{\prime\prime}\ov{\omega}+im\tilde{a}\psi.
$$
In particular, we have $\pi^*(\ov{\partial} A)=\tilde{a}^{\prime\prime}\ov{\omega}\otimes\omega^m$. Since
\begin{align*}
\d a&=X(a)\omega_1+H(a)\omega_2+V(a)\psi,\\
\d (Va)&=X(V(a))\omega_1+H(V(a))\omega_2-m^2a\psi,
\end{align*}
we obtain
$$
\tilde{a}^{\prime\prime}=\frac{1}{2}\left(XVa/m-Ha\right)+\frac{i}{2}\left(HVa/m+Xa\right).
$$
We also have
$$
\pi^*\left(\theta-i\star_g\theta\right)=\left(\theta+i V\theta\right)\ov{\omega}.
$$
Hence~\eqref{eq:weakhol1} is equivalent to 
$$
\tilde{a}^{\prime\prime}-\left(\frac{m-1}{2}\right)(\theta+iV\theta)(V(a)/m+i a)=0.
$$
Taking the real part gives~\eqref{eq:realcond}. 
\end{proof}
\begin{Remark}
Recall that a torsion-free connection on $TM$ preserving a conformal structure $[g]$ is called a~\textit{Weyl connection} or \textit{conformal connection}. More precisely, $\nabla$ preserves $[g]$ if for some (and hence any) $g\in[g]$, there exists a $1$-form $\theta$, so that
$$
\nabla g=2\theta\otimes g.
$$
\end{Remark}

\begin{Remark}[The case $m=1$]\label{rmk:casem=1}
We could also consider the case $\lambda=a-V\theta$ with $a$ representing a differential of degree $m=1$, that is, a $(1,\! 0)$-form. We exclude this case since it corresponds to the case where $A$ vanishes identically by defining $\theta^{\prime}=Va$ and considering $\lambda^{\prime}=-V(\theta^{\prime}-\theta)=\lambda$. Flows defined by $\lambda=-V\theta=$ were studied previously under the name $W$-flows as they arise naturally by reparametrising the geodesics of a Weyl connection, see~\cite{W2}. In particular in~\cite[Theorem 5.2]{W2} it is proved that $W$-flows are Anosov provided $K_g-\delta_g\theta<0$. A simple computation gives that $\mathbb{K}=K_g-\delta_g \theta$
hence we recover~\cite[Theorem 5.2]{W2} by applying Proposition~\ref{prop:basic}. In particular, we see that if $A$ is a holomorphic $1$-form and $g$ satisfies $K_g<0$, then the associated thermostat flow is Anosov. 
\end{Remark} 
We now want to apply Theorem~\ref{thm:workhorse} to the case $\lambda=a-V\theta$  for some good choice of $p$. 
\begin{Lemma}\label{lem:wanggeneralized}
Suppose $\lambda=a-V\theta$ and take $p=Va/m+\theta$. Then $\kappa_p\equiv -1$ if and only if the following two equations are identically satisfied
\begin{align}
\label{eq:genwang}K_g&=-1+\delta_g\theta+(m-1)|A|^2_g,\\
\label{eq:weakhol}\ov{\partial}A&=\left(\frac{m-1}{2}\right)\left(\theta-i\star_g\theta\right)\otimes A.
\end{align}
\end{Lemma}
\begin{proof}
Taking $p=Va/m+\theta$ gives
\begin{align*}
\kappa_p&=\kappa+Fp+p(p-V\lambda)=K_g-H\lambda+\lambda^2+Fp+p(p-V\lambda)\\
&=K_g-Ha+HV\theta+a^2-2aV\theta+(V\theta)^2+(X+(a-V\theta)V)(Va/m+\theta)+p(p-V\theta)\\
&=K_g+HV\theta+X\theta-(m-1)\left(a^2+(Va)/m^2\right)+\left(XVa/m-Ha-(m-1)(\theta Va/m-aV\theta)\right)\\
&=K_g-\delta_g\theta-(m-1)|A|^2_g+\frac{1}{m}\left(XVa-mHa-(m-1)(\theta Va-maV\theta)\right),
\end{align*}
where we have used~\eqref{eq:idcodiff},~\eqref{eq:normdiff} and $VVa=-m^2 a$ as well as $VV\theta=-\theta$. 
Using Lemma~\ref{lem:transnotthomasgabriel} we see that $\kappa_p\equiv -1$ provided~\eqref{eq:genwang} and~\eqref{eq:weakhol} are identically satisfied. Conversely, suppose $\kappa_p\equiv -1$. Since $K_g-\delta_g\theta-(m-1)|A|^2_g$ is constant along the fibres of $SM \to M$, we obtain 
$$
0=V\kappa_p=\left(\frac{1-m}{m}\right)\Big(HVa+mXa-(m-1)\left(m\theta a+V(\theta)V(a)\right)\Big).
$$
Lemma~\ref{lem:transnotthomasgabriel} and Remark~\ref{rmk:imagcond} therefore imply that~\eqref{eq:weakhol} must hold. Hence we also identically have
$$
\kappa_p=-1=K_g-\delta_g\theta-(m-1)|A|^2_g, 
$$
which is equivalent to~\eqref{eq:genwang}. 
\end{proof}
Combining Theorem~\ref{thm:workhorse} and Lemma~\ref{lem:wanggeneralized} we thus immediately obtain:
\begin{Corollary}\label{cor:domsplit}
Let $(g,A,\theta)$ be a triple on $M$ satisfying~\eqref{eq:genwang} and~\eqref{eq:weakhol}. Then the associated thermostat flow admits a dominated splitting. 
\label{cor:domsplitting}
\end{Corollary}

We also observe:
\begin{Proposition} Consider a pair $(g,A)$ with $A$ holomorphic and $K_g<0$. Then the associated thermostat flow has a dominated splitting. Moreover, for $m=2$, the flow is Anosov.
\end{Proposition}
\begin{proof}The fact that there is a dominated splitting follows from $\kappa_{p}<0$.  For $m=2$ we note that 
$$
\kappa_{p}=K_g-|A|^2_g=K_g-a^{2}-(Va)^{2}/4.
$$
Thus 
$\kappa_{p}+(Va)^{2}/4<0$ and the Anosov property follows from Theorem \ref{thm:workhorse}.
\end{proof}

\subsection{Parametrising thermostat flows arising from differentials}

It turns out that the thermostat flows defined by triples $(g,A,\theta)$ satisfying~\eqref{eq:genwang} and~\eqref{eq:weakhol} can be parametrised in terms of complex geometric data. For $m\geqslant 2$ define the (smooth) complex line bundle $L_m:=\Lambda^2(TM)^{(m-1)/2}\otimes \C$. 
\begin{Lemma}\label{lem:conlinebundstruc}
There exists a canonical bijection between the following sets:
\begin{itemize}
\item[(i)] the holomorphic line bundle structures on $L_m$;
\item[(ii)] the $[g]$-conformal connections on $TM$.   
\end{itemize}
\end{Lemma}

Before we prove Lemma~\ref{lem:conlinebundstruc}, we first recall some basic facts about conformal connections. Let us fix a Riemannian metric $g \in [g]$. It follows from Koszul's identity that the $[g]$-conformal connections are of the form 
$$
{}^{(g,\theta)}\nabla={}^g\nabla+g\otimes \theta^{\sharp}-\theta\otimes \mathrm{Id}-\mathrm{Id}\otimes \theta
$$
where $\theta \in \Omega^1(M)$, ${}^g\nabla$ denotes the Levi-Civita connection of $g$ and $\theta^{\sharp}$ the $g$-dual vector field of $\theta$. Moreover, for $u \in C^{\infty}(M)$, we have~\cite[Theorem 1.159]{Besse}
$$
{}^{\exp(2u)g}\nabla={}^g\nabla-g\otimes{}^g\nabla u+\d u\otimes \mathrm{Id}+\mathrm{Id}\otimes \d u
$$
from which one easily computes
$$
{}^{(\exp(2u)g,\theta+\d u)}\nabla={}^{(g,\theta)}\nabla.
$$
Since ${}^{(g,\theta)}\nabla g=2\,\theta\otimes g$ and ${}^{(g,\theta)}\nabla \mathrm{e}^{2u}g=2\,(\theta+\d u)\otimes \mathrm{e}^{2u}g$, we conclude that the $[g]$-conformal connections are in one-to-one correspondence with~\textit{Weyl structures}, where by a Weyl structure we mean an equivalence class $[g,\theta]$ subject to the equivalence relation
$$
(g,\theta)\sim (\hat{g},\hat{\theta})\quad \iff \quad \hat{g}=\mathrm{e}^{2u}g \;\; \text{and}\;\; \hat{\theta}=\theta+\d u
$$
for $u \in C^{\infty}(M)$. For later usage we also record that the symmetric part of the Ricci curvature of ${}^{(g,\theta)}\nabla$ satisfies
$$
\mathrm{Sym}\,\mathrm{Ric}\left({}^{(g,\theta)}\nabla\right)=\left(K_g-\delta_g\theta\right)g. 
$$

\begin{proof}[Proof of Lemma~\ref{lem:conlinebundstruc}]
Let $\ov{\partial}_{L_m} : \Gamma(M,L_m) \to \Omega^{0,1}(M,L_m)$ be a holomorphic line bundle structure on $L_m$. Observe that $(\det g)^{-(m-1)/4}$ is a non-vanishing section of $L_m$, hence 
$$
(\det g)^{(m-1)/4}\otimes\ov{\partial}_{L_m}(\det g)^{-(m-1)/4}
$$
is a $(0,\! 1)$-form on $M$. Thus there exists a unique $1$-form $\theta$ on $M$ so that
$$
\ov{\partial}_{L_m}(\det g)^{-(m-1)/4}=-\left(\frac{m-1}{2}\right)\left(\theta-i\star_g\theta\right)\otimes (\det g)^{-(m-1)/4}.
$$
If we instead consider the metric $\hat{g}=\mathrm{e}^{2u}g$ for $u \in C^{\infty}(M)$, then we obtain
$$
\ov{\partial}_{L_m}(\det \hat{g})^{-(m-1)/4}=-\left(\frac{m-1}{2}\right)\left(\hat{\theta}-i\star_g\hat{\theta}\right)\otimes (\det \hat{g})^{-(m-1)/4}
$$
with $\hat{\theta}=\theta+\d u$. It follows that $\ov{\partial}_{L_m}$ defines a Weyl structure on $M$. Moreover, if two holomorphic line bundle structures $\ov{\partial}_{L_m}$ and $\ov{\partial}^{\prime}_{L_m}$ on $L_m$ determine the same Weyl structure $[g,\theta]$, then they satisfy
$$
\ov{\partial}_{L_m}(\det g)^{-(m-1)/4}=\ov{\partial}^{\prime}_{L_m}(\det g)^{-(m-1)/4}
$$
and hence also $\ov{\partial}_{L_m}=\ov{\partial}^{\prime}_{L_m}$. 

Conversely, let ${}^{(g,\theta)}\nabla$ be a $[g]$-conformal connection, then
$$
{}^{(g,\theta)}\nabla \left(\det g\right)^{-(m-1)/4}=-\left(m-1\right)\theta\otimes \left(\det g\right)^{-(m-1)/4}. 
$$
Extending ${}^{(g,\theta)}\nabla$ complex linearly, we obtain a  connection on the complex line bundle $L_m$ whose curvature form is (since $\dim_{\C} M=1$) an $\mathrm{End}(L_m)$-valued $(1,\! 1)$-form on $M$. Thus, standard results imply (c.f.~\cite[Prop.~1.3.7]{Koba}) that there exists a unique holomorphic line bundle structure $\ov{\partial}_{L_m}$ on $L_m$ so that $\ov{\partial}_{L_m}={}^{(g,\theta)}\nabla^{(0,1)}$. Finally, we have
\begin{align*}
{}^{(g,\theta)}\nabla^{(0,1)}\left(\det g\right)^{-(m-1)/4}&=-\left(\frac{m-1}{2}\right)\left(\theta-i\star_g\theta\right)\otimes (\det g)^{-(m-1)/4}\\
&=\ov{\partial}_{L_m}(\det g)^{-(m-1)/4}.
\end{align*}
Therefore, the Weyl structure determined by $\ov{\partial}_{L_m}$ is $[g,\theta]$, thus proving the claim. 
 \end{proof}
 
 Given a section $P$ of $L_{m}\otimes K_{M}^{m}$ we can define
 \[|P|_{g}^2:=|A|^{2}_{g}\]
 where $A:=\left(\det g\right)^{(m-1)/4}\otimes P$. It is straightforward to check that the quadratic form
 \[\mathbb{P}:=|P|^{2}_{g}g\]
 only depends on $[g]$.
 
We now have:
\begin{Proposition}
Let $m\geqslant 2$. On a compact oriented surface $M$ with $\chi(M)<0$ the following sets are in one-to-one correspondence: 
\begin{itemize}
\item[(i)] the triples $(g,A,\theta)$ consisting of a Riemannian metric $g$, a differential $A$ of degree $m$ and a $1$-form $\theta$ such that 
$$
K_g=-1+\delta_g\theta+(m-1)|A|^2_g\quad \text{and}\quad \ov{\partial} A=\left(\frac{m-1}{2}\right)\left(\theta-i\star\theta\right)\otimes A;
$$
\item[(ii)] the triples $([g],\ov{\partial}_{L_m},P)$ consisting of a conformal structure $[g]$, a holomorphic line bundle structure $\ov{\partial}_{L_m}$ on $L_m$ and a holomorphic section $P$ of $L_m\otimes K_M^m$ having the property that the symmetric part of the Ricci curvature of the conformal connection associated to $\ov{\partial}_{L_m}$ plus $(1-m)\mathbb{P}$ is negative definite.
\end{itemize}
\end{Proposition}
\begin{proof}
Suppose $(g,A,\theta)$ is a triple satisfying
$$
K_g=-1+\delta_g\theta+(m-1)|A|^2_g \quad \text{and}\quad \ov{\partial}A=\left(\frac{m-1}{2}\right)\left(\theta-i\star_g\theta\right)\otimes A.
$$
We equip $L_m$ with the holomorphic line bundle structure induced by the conformal connection ${}^{(g,\theta)}\nabla$. Define $P:=\left(\det g\right)^{-(m-1)/4}\otimes A$, then $P$ is a holomorphic section of $L_m\otimes K_M^m$. Indeed, we compute
\begin{align*}
\ov{\partial}P&=\ov{\partial}_{L_m}\left((\det g)^{-(m-1)/4}\right)\otimes A+\left(\det g\right)^{-(m-1)/4}\otimes \ov{\partial}_{K_M}A\\
&=-\left(\frac{m-1}{2}\right)\left(\theta-i\star_g\theta\right)\otimes P+\left(\frac{m-1}{2}\right)\left(\theta-i\star_g\theta\right)\otimes P\\
&=0.
\end{align*}
In addition, we observe that the symmetric part of the Ricci curvature of ${}^{(g,\theta)}\nabla$ satisfies
$$
\mathrm{Sym}\;\mathrm{Ric}\left({}^{(g,\theta)}\nabla\right)+(1-m)\mathbb{P}=\left(K_g-\delta_g\theta+(1-m)|A|_{g}^{2}\right)g=-g
$$
which is obviously negative definite. Clearly, the just described map from the first set of triples into the second set of triples is injective. 

Conversely, suppose $L_m$ is equipped with a holomorphic line bundle structure $\ov{\partial}_{L_m}$ and let $P$ be a holomorphic section of $L_m\otimes K_M^m$. Assume furthermore that the symmetric part of the Ricci curvature of the conformal connection associated to $\ov{\partial}_{L_m}$ plus $(1-m)\mathbb{P}$ is negative definite. We will next use these data to construct a triple $(g,A,\theta)$ solving the above equations. Let $g_0 \in [g]$ denote the hyperbolic metric in the conformal equivalence class and define
$$
A_0:=\left(\det g_0\right)^{(m-1)/4}\otimes P.
$$
Note that $(\det g_0)^{(m-1)/4}$ is a non-vanishing section of $L_m^{-1}$ and hence $A_0$ is a section of $K_M^m$. Since $P$ is holomorphic it follows that there exists a unique $1$-form $\theta_0$ on $M$ such that
$$
\ov{\partial}A_0=\left(\frac{m-1}{2}\right)\left(\theta_0-i\star\theta_0\right)\otimes A_0. 
$$
Now make the Ansatz $g=\mathrm{e}^{2u}g_0$ for $u \in C^{\infty}(M)$ and $A=\left(\det g\right)^{(m-1)/4}\otimes P=A_0\mathrm{e}^{u(m-1)}$. Then
$$
\ov{\partial}A=\left(\frac{m-1}{2}\right)\left(\theta-i\star\theta\right)\otimes A,
$$
where $\theta=\theta_0+\d u$. Since
\begin{equation}\label{eq:gaussconfchange}
K_{\exp(2u)g}=\mathrm{e}^{-2u}\left(K_g-\Delta_g u\right),
\end{equation}
where $\Delta_g=-\left(\delta_{g}\d+\d\delta_{g}\right)$, we obtain
$$
\mathrm{e}^{-2u}\left(-1-\Delta u\right)=-1+\mathrm{e}^{-2u}\delta\left(\theta_0+\d u\right)+(m-1)\mathrm{e}^{-2u}|A_0|^2,
$$
where now all norms and operators are with respect to $g_0$. This simplifies to become an algebraic equation for $u$
$$
\mathrm{e}^{2u}-(m-1)|A_0|^2=1+\delta \theta_0.
$$
Clearly, this equation uniquely determines $u$ provided $1+\delta\theta_0+(m-1)|A_{0}|^{2}$ is positive. Note that this happens if and only if 
$$
(-1-\delta\theta_0+(1-m)|A_{0}|^{2})g_0=\mathrm{Sym}\,\mathrm{Ric}\left({}^{(g_0,\theta_0)}\nabla\right)+(1-m)\mathbb{P}
$$ is negative definite, but ${}^{(g_0,\theta_0)}\nabla$ is just the conformal connection induced by $\ov{\partial}_{L_m}$. Finally, by construction, the triple associated to $(g,A,\theta)$ is $([g],\ov{\partial}_{L_m},P)$.  
\end{proof}

\begin{Remark}[W-Flows]
The W-Flows of Wojtkowski~\cite{W2} are also covered by the thermostat flows defined by triples $(g,A,\theta)$ satisfying~\eqref{eq:genwang} and \eqref{eq:weakhol} in the case where the conformal connection ${}^{(g,\theta)}\nabla$ defining the W-flow has negative definite symmetric Ricci curvature, that is, satisfies $(K_g-\delta_g\theta)<0$. Indeed, suppose the pair $(g,\theta)$ satisfies $(K_g-\delta_g\theta)<0$. Let 
$
u=\frac{1}{2}\ln\left(\delta_g\theta-K_g\right)
$
and consider $(\hat{g},\hat{\theta})=(\mathrm{e}^{2u}g,\theta+\d u)$. Then the pairs $(g,\theta)$ and $(\hat{g},\hat{\theta})$ define the same conformal connection and hence equivalent W-flows. Using ~\eqref{eq:gaussconfchange} and the identity $\delta_{\exp(2u)g}=\mathrm{e}^{-2u}\delta_g$ for the co-differential acting on $1$-forms, we compute
$$
K_{\hat{g}}-\delta_{\hat{g}}\hat{\theta}=\left(\frac{1}{\delta_g\theta-K_g}\right)\left(K_g-\Delta_g u\right)-\left(\frac{1}{\delta_g\theta-K_g}\right)\delta_g\left(\theta+\d u\right)=-1. 
$$
Hence the triple $(\hat{g},0,\hat{\theta})$ satisfies~\eqref{eq:genwang} and \eqref{eq:weakhol}. In particular, we see that the geodesic flow of metrics of negative Gauss curvature also fit into our family of flows. 
\end{Remark}

\section{The case of holomorphic differentials}

We have seen that a triple $(g,A,\theta)$ solving~\eqref{eq:genwang} and~\eqref{eq:weakhol} yields a holomorphic section of $L_m\otimes K_M^m$ with respect to some appropriate holomorphic line bundle structure on $L_m$. We now restrict to the case where the differential $A$ is already holomorphic so that we obtain the coupled vortex equations
$$
K_g=-1+(m-1)|A|^2_g\quad \text{and} \quad \ov{\partial} A=0.
$$
\subsection{Anosov flows}
It is possible to upgrade Corollary \ref{cor:domsplit} in the case where $A$ is holomorphic as follows:

\begin{Theorem}\label{thm:anosovholdif} Let $(g,A)$ be a pair satisfying the coupled vortex equations $\bar{\partial}A=0$ and
$K_g=-1+(m-1)|A|^{2}_{g}$. Then the associated thermostat flow is Anosov.
\label{thm:upgradeanosov}
\end{Theorem}

\begin{proof} We already know that there is a dominated splitting, so taking into account Remark \ref{rem:upgrade}, the strategy will be to show that $r^{u}>0$ and $r^{s}<0$. We will do this using the following lemma.

\begin{Lemma} Let $(g,A)$ be a pair satisfying the coupled vortex equations $\bar{\partial}A=0$ and
$K_g=-1+(m-1)|A|^{2}_{g}$. Then $-1\leqslant K_g<0$.
\label{lemma:analogue_BH}
\end{Lemma}

\begin{proof} The proof is quite similar to the proof of \cite[Proposition 3.3]{BH}, the reader may also compare with~\cite[Theorem 5.1]{DM}. The claim is obviously correct if $A$ vanishes identically, hence we assume this not to be the case. We first prove the inequality $K_g\leqslant 0$. As before let $g_0$ denote the hyperbolic metric in the conformal equivalence class of $g$ and write $g=\mathrm{e}^{2u}g_0$ for $u \in C^{\infty}(M)$. Using 
\begin{equation}\label{eq:confchangeids}
K_g=\mathrm{e}^{-2u}\left(-1-\Delta u\right)\quad \text{and}\quad |A|^2_{g}=\mathrm{e}^{-2mu}{|A|^2_{g_0}}
\end{equation} gives
\begin{equation}\label{eq:pdeconffactor}
1+\Delta u=\mathrm{e}^{2u}-(m-1)\mathrm{e}^{-2(m-1)u}\alpha,
\end{equation}
where we write $\alpha=|A|^2_{g_0}$. The inequality $K_g\leqslant 0$ is equivalent to
\begin{equation}\label{eq:gausscurvinequality}
(m-1)\mathrm{e}^{-2mu}\alpha \leqslant 1
\end{equation}
and is clearly satisfied at the points where $A$ vanishes. Therefore, taking the logarithm of~\eqref{eq:gausscurvinequality}, we see that $K_g\leqslant 0$ follows from the non-negativity of the smooth function
$$
f=2mu-\log(m-1)-\log \alpha,
$$
which is defined on the open set $M^{\circ}:=\left\{x \in M : A(x)\neq 0\right\}$. Note that using $f$ the equation~\eqref{eq:pdeconffactor} becomes
\begin{equation}\label{eq:pdeconffactor2}
1+\Delta u=\mathrm{e}^{2u}(1-\mathrm{e}^{-f}). 
\end{equation}
As $M$ is compact, the Gauss curvature $K_g$ attains its maximum at some point $x_0$ and moreover $x_0 \in M^{\circ}$. Consequently, the function $f$ attains its infimum at $x_0$. A straightforward calculation gives $\Delta\log \alpha=-2m$, where we use that $A$ is holomorphic. At the minimum $x_0$ of $f$ we thus obtain
\begin{equation}\label{eq:deltaf}
0\leqslant \Delta f(x_0)=2m\left(1+\Delta u(x_0)\right)=2m\,\mathrm{e}^{2u(x_0)}\left(1-\mathrm{e}^{-f(x_0)}\right),
\end{equation}
where we have used~\eqref{eq:pdeconffactor2}. It follows that $f(x_0)\geqslant 0$ and hence $f\geqslant 0$ on all of $M^{\circ}$. This shows that $K_g\leqslant 0$. It order to prove $K_g<0$, we first remark that the function $f-1+\mathrm{e}^{-f}$ is non-negative on $M^{\circ}$. Consequently, ~\eqref{eq:deltaf} gives
$$
\Delta_g f\leqslant 2m f, 
$$
where $\Delta_g=\mathrm{e}^{-2u}\Delta$ denotes the Laplacian with respect to $g$. In particular, it follows that for every point $x \in M^{\circ}$ there exists a constant $c>0$, an $x$-neighbourhood $U_x$ and a flat metric $g_0$ on $U_x$ which lies in the conformal equivalence of $g$, so that
$$
\left(\Delta_{g_0}-c\right)f \leqslant 0
$$
on $U_x$. Therefore, by applying the strong maximum principle~\cite[Theorem~3.5]{GT} to the operator $\Delta_{g_0}-c$, it follows that if $f$ vanishes at some point in $U_x$, then it vanishes on all of $U_x$ and consequently on $M^{\circ}$. Since $A$ is holomorphic, its zeros are isolated and hence $M^{\circ}$ is dense in $M$. Since $K_g$ is continuous we conclude that if $K_g$ vanishes at some point on $M$, then it vanishes identically on $M$, but this possibility is excluded by the Gauss--Bonnet theorem. 
\end{proof}
\begin{Remark}\label{rmk:exsolpde}
From~\eqref{eq:pdeconffactor} we see that $u$ solves a PDE of the form $\Delta u=G(x,u)$ where
$$
G(x,u)=-1+\mathrm{e}^{2u}-(m-1)\mathrm{e}^{-2(m-1)u}\alpha(x). 
$$
Since $\alpha\geqslant 0$ we have $G(x,u)\leqslant -1+\mathrm{e}^{2u}$ and hence $G(x,u)<0$ for $u<0$. On the other hand, for $u> \sup_{x \in M} \frac{1}{2}\log(1+(m-1)\alpha(x))\geqslant 0$ we get
$$
G(x,u)>-1+\mathrm{e}^{2u}-(m-1)\alpha(x)>0.
$$
Since
$$
\frac{\partial G}{\partial u}(x,u)=2\alpha(x)(m-1)^2\mathrm{e}^{-2(m-1)u}+2\mathrm{e}^{2u}>0
$$
standard quasi-linear elliptic PDE methods (see for instance~\cite[Proposition 1.9]{taylorIII}) imply that~\eqref{eq:pdeconffactor} has a unique smooth solution $u$ for every smooth non-negative function $\alpha$. Consequently, for every holomorphic differential $A$ on $(M,[g])$ we obtain a unique solution $(g,A)$ to the coupled vortex equations $K_g=-1+(m-1)|A|^2_g$ and $\ov{\partial} A=0$.
\end{Remark}

We now show that $r^u>0$ (the proof that $r^{s}<0$ is similar). Set $h=r^u-V(a)/m$. Then $h$ satisfies
\[F(h)+h^{2}+hB-1=0,\]
where 
$$B:=\frac{(2-m)}{m}V(a).$$
Given $(x,v)\in SM$, consider for each $R>0$, the unique solution $h_{R}$ to the Riccati equation
along $\phi_{t}(x,v)$:
\[\dot{h}+h^2+hB-1=0\]
satisfying $h_{R}(-R)=\infty$. Using \eqref{eq:hopflimit} we derive
\begin{equation}
r^{u}(x,v)=\lim_{R\to\infty}h_{R}(0)+V(a)/m.
\label{eq:hopflimith}
\end{equation}
Let $c:=\max_{(x,v)}|B(x,v)|$ and $\ell:=\frac{\sqrt{c^{2}+4}-c}{2}$. 
If we let $f_{R}:=h_{R}-\ell$, then $f_{R}$ solves
\begin{equation}
\dot{f}+wf=q,
\label{eq:linear}
\end{equation}
where $w:=f_{R}+B+2\ell$ and $q:=-\ell^{2}-B\ell+1$. Observe that $q\geqslant 0$ by our definitions of $c$ and $\ell$.
We can solve the inhomogeneous linear equation \eqref{eq:linear} and use that $q\geqslant 0$ to derive
$f_{R}(t)\geqslant 0$ and thus $h_{R}(t)\geqslant \ell$. By taking limits, 
and using \eqref{eq:hopflimith}, we obtain
\[r^{u}(x,v)\geqslant \ell +V(a)/m.\]
By Lemma \ref{lemma:analogue_BH} we have $c<(m-2)/\sqrt{m-1}$ and $V(a)/m>-1/\sqrt{m-1}$. Thus
\[r^{u}\geqslant  \frac{\sqrt{c^{2}+4}-c}{2}-\frac{1}{\sqrt{m-1}}>0\]
as desired.
\end{proof}

\begin{Remark} As we have seen, Corollary \ref{cor:domsplitting} asserts that given a triple $(g,A,\theta)$ satisfying ~\eqref{eq:genwang} and~\eqref{eq:weakhol}, the associated thermostat flow has a dominated splitting.
When $\theta=0$, Theorem \ref{thm:upgradeanosov} tells us that we can do better and in fact the thermostat flow is Anosov. At the ``other end'', that is, when $A=0$, we also know by Proposition \ref{prop:basic} that the thermostat flow is also Anosov (in this case $\mathbb{K}=K_{g}-\delta_{g}\theta=-1$).
These two ``ends'' are Anosov for different reasons, connected with the discussion in Remark \ref{rem:upgrade}.
In the case $\theta=0$, as we have just seen, one uses that $r^{u}>0$, that is, the first case in \eqref{eq:alternative}. In the case $A=0$, we use the second case in (\ref{eq:alternative}). It is conceivable that the thermostat flow is always Anosov for any triple $(g,A,\theta)$ satisfying ~\eqref{eq:genwang} and~\eqref{eq:weakhol}, but at the time of writing it is not at all clear how to prove this.
It should be noted that for the special case of the geodesic flow it is well known that a dominated splitting must be Anosov. We can see this fairly quickly using quadratic forms as follows. Suppose $r^{u,s}:SM\to\R$ are two continuous functions such that $Xr^{u,s}+[r^{s,u}]^{2}+K_{g}=0$ and $r^{u}-r^{s}\neq 0$ everywhere.
Define
\[Q=2y\dot{y}-([r^{u}]^{2}+[r^{s}]^{2})y^2.\]
Then a calculation shows
\[\dot{Q}=(\dot{y}-r^{u}y)^{2}+(\dot{y}-r^{s}y)^{2}>0\]
unless $y=\dot{y}=0$. Hence by Proposition \ref{ppn:sufsplit} the geodesic flow is Anosov.
\end{Remark}

\subsection{Dissipation and volume} We will now prove the following result stated in the introduction. 

\begin{Theorem}Let $(g,A)$ be a pair satisfying the coupled vortex equations $\bar{\partial}A=0$ and
$K_g=-1+(m-1)|A|^{2}_{g}$.  Then the associated thermostat flow preserves an absolutely continuous measure if and only if $A$ vanishes identically.
\label{thm:novolume}
\end{Theorem}

\begin{proof} Since the flow is of class $C^{\infty}$ and Anosov, an application of the smooth Liv\v sic theorem  \cite[Corollary 2.1]{LMM}
shows that $\phi_t$ preserves an absolutely continuous measure
if and only if $\phi_t$ preserves a smooth volume form. 

We write the volume form as $\mathrm{e}^{-u}\Theta$
for some real-valued function $u$ on $SM$. Thus, using~\eqref{lie1}, we obtain
$$
L_{F}\left(\mathrm{e}^{-u}\Theta\right)=-\mathrm{e}^{-u}F(u)\Theta+\mathrm{e}^{-u}V(a)\Theta=(-Fu+Va)\mathrm{e}^{-u}\Theta.
$$
Hence the claim follows by showing that if $u$ solves $Fu=Va$, then $a$ vanishes identically. In order to show this we use the following $L^2$ identity proved in \cite[Equation (5)]{JP} which is in turn an extension of an identity in \cite{SU} for geodesic flows. The identity holds for arbitrary thermostats $F=X+\lambda V$. If we let $H_{c}:=H+cV$ where $c:SM\to\re$ is any smooth function then
\begin{equation}
2\langle H_{c}u,VFu\rangle=\|Fu\|^{2}+\|H_{c}u\|^{2}-\langle Fc+c^{2}+K_g-H_{c}\lambda+\lambda^{2},(Vu)^{2}\rangle,
\label{eq:l2}
\end{equation}
where $u$ is any smooth function. All norms and inner products are $L^{2}$ with respect to the volume form $\Theta$.

In our case $\lambda=a$ and a calculation shows that if we pick $c=V(a)/m$, then
$$
Fc+c^{2}+K_g-H_{c}\lambda+\lambda^{2}=K_g+(1-m)|A|_{g}^{2}=-1,
$$ 
hence for this choice of $c$, \eqref{eq:l2} simplifies to
\begin{equation}
2\langle H_{c}u,VFu\rangle=\|Fu\|^{2}+\|H_{c}u\|^{2}+\|Vu\|^{2}.
\label{eq:l22}
\end{equation}
If $Fu=Va$, then $VFu=-m^{2}a$ and we compute using that $X$ and $H$ preserve $\Theta$ and that
$XVa-mHa=0$:
\begin{align*}
2\left\langle H_{c}u,VFu\right\rangle&=-2m^{2}\langle Hu,a\rangle-2m^{2}\langle cVu,a\rangle\\
&=2m^{2}\langle u,Ha\rangle -2m^{2}\langle cVu,a\rangle\\
&=-2m^{2}\langle Xu,V(a)/m\rangle-2m^{2}\langle cVu,a\rangle\\
&=-2m\|Va\|^{2},
\end{align*}
where the last equation is obtained using that $Xu=Va-aVu$ and $c=V(a)/m$. Inserting this back into \eqref{eq:l22}, we see that the equality obtained can only hold if $Va$ and hence $a$ vanishes identically.
\end{proof}

\section{The cases $m=2$ and $m=3$} 
\label{section:2-3}In this section we consider the special cases of $m=2,3$ and their peculiarities. These flows have appeared in different contexts and for different reasons and in this section we explain these features.

\subsection{The case $m=2$} Consider a pair $(g,A)$ where $A$ is a quadratic differential with $\bar{\partial}A=0$ and $K_g=-1+|A|^{2}_{g}$. By Theorem \ref{thm:upgradeanosov}, the associated thermostat flow is Anosov.
These flows have the distinctive feature that their weak bundles are of class $C^{\infty}$. Indeed for this case $p=V(a)/2$, $\kappa_{p}=-1$ and equation \eqref{eq:ric2} reduces to
\[Fh+h^{2}-1=0.\]
From this we clearly see that $r^{u,s}=\pm 1+V(a)/2$ and hence the weak bundles
\[\mathbb{R}F\oplus\mathbb{R}(H+r^{s,u}V)\]
are smooth. This class of thermostats flows was first considered in \cite{P07}, where the coupled vortex equations for $m=2$ were derived assuming that the weak foliations were smooth. Theorem 4.6 in \cite{Ghy2} asserts that
a smooth Anosov flow on a closed 3-manifold with weak stable and unstable foliations of class $C^{1,1}$,
is smoothly orbit equivalent to a suspension or to a {\it quasi-fuchsian flow} as described in \cite[Th\'eor\`eme B]{Ghy1}. (In our case, since we are working with circles bundles
the latter alternative holds.) A quasi-fuchsian flow $\psi$ depends on a pair of points $([g_1],[g_2])$
in Teichm\"uller space, has smooth weak stable foliation $C^{\infty}$-conjugate to the
weak stable foliation of the constant curvature metric $g_1$ and smooth weak unstable foliation
$C^{\infty}$-conjugate to the weak unstable foliation of the constant curvature metric $g_2$.
Moreover, $\psi$ preserves a volume form if and only if $[g_1]=[g_2]$.
The analogous result on the thermostat side is provided by
Theorem \ref{thm:novolume} which asserts that the thermostat flow preserves a volume form iff $A=0$.
It is an interesting question (first raised in \cite{P07}) to decide if the thermostat flows originating from the coupled vortex equations $\bar{\partial}A=0$, $K_g=-1+|A|^{2}_{g}$ describe all possible quasi-fuchsian flows $\psi$.

\subsection{The case $m=3$}
Let now $(g,A,\theta)$ be a triple on $M$ satisfying~\eqref{eq:genwang} and~\eqref{eq:weakhol} with $A$ being a cubic differential. The connection form of the Levi-Civita connection on the tangent bundle $TM$
is
$$
\begin{pmatrix} 0 & -\psi \\ \psi & 0\end{pmatrix}.
$$
We define a $1$-form on $SM$ with values in $\mathfrak{gl}(2,\R)$
\begin{multline*}
\Upsilon=(\Upsilon^i_j)=\begin{pmatrix} 0 & -\psi \\ \psi & 0 \end{pmatrix}\\
+\begin{pmatrix} (V(a)/3-\theta)\omega_1-(a+V(\theta))\omega_2 & -(V(\theta)+a)\omega_1+(\theta-V(a)/3)\omega_2\\
(V(\theta)-a)\omega_1-(\theta+V(a)/3)\omega_2 & -(\theta+V(a)/3)\omega_1+(a-V(\theta))\omega_2
\end{pmatrix}.
\end{multline*}
It is a consequence of the equivariance properties
$$
VVa=-9 a, \quad VV\theta=-\theta, \quad L_{V}\omega_1=\omega_2, \quad \text{and}\quad L_{V}\omega_2=-\omega_1
$$
that the $1$-form $\Upsilon$ is the connection $1$-form of a unique (torsion-free) connection $\nabla$ on the tangent bundle $TM$. Moreover, since the interior product $i_F\Upsilon^2_1$ vanishes identically for $\lambda=a-V\theta$, it follows that the geodesics of the connection $\nabla$ can be reparametrised to agree with the projections to $M$ of the orbits of the thermostat flow defined by $\lambda$, see~\cite[Lemma 3.1]{MetPat2} for details. Moreover, if $\theta$ is closed the connection $\nabla$ admits an interpretation as a Lagrangian minimal surface, see~\cite{Met}. If $A$ is holomorphic so that $\theta$ vanishes identically, then the connection $\nabla$ defines a properly convex projective structure on $M$, see the work of Labourie~\cite{Lab} and~\cite{Met2,Met}. This means that the universal cover $\Omega$ of $M$ is a properly convex open subset of the real projective plane $\mathbb{RP}^2$ for which there exists a discrete group $\Gamma$ of projective transformations which acts cocompactly on $\Omega$ and so that $M=\Omega/\Gamma$. Thus, $(\Omega,\Gamma)$ is a divisible convex set. Moreover, the segments of the projective lines $\mathbb{RP}^1$ contained in $\Omega$ project to $M$ to agree with the (unparametrised) geodesics of $\nabla$. The universal cover $\Omega$ being a convex set, it is equipped with the Hilbert metric. The geodesic flow of the Hilbert metric descends to $\mathbb{S}M$ and by a result of Benoist~\cite{BConvDivI}, is Anosov if and only if $\Omega$ is strictly convex. In~\cite{BConvDivI}, it is also shown that a divisible convex set is strictly convex if and only if the group dividing it is word-hyperbolic. Since the fundamental group of a closed surface of negative Euler characteristic is word-hyperbolic, it thus follows from known results that the thermostat flow associated to a holomorphic cubic differential is a reparametrisation of an Anosov flow. However, since the Anosov property is invariant under reparametrisation of the flow, we conclude that the thermostat flow associated to a holomorphic cubic differential is Anosov, which is the statement of our Theorem~\ref{thm:anosovholdif} for the special case $m=3$.  

\section{Regularity of weak foliations} As we previously mentioned, the case of $m=2$ has the distinctive feature of having weak bundles of class $C^{\infty}$. It is natural to ask what happens for $m\geq 3$.
One approach to this question would be to compute the Godbillon--Vey invariant  following \cite{P07}.
Unfortunately for $m\geq 3$ this calculation does not yield information conducive to an answer.
However, for the case $m$ odd, we can use reversibility of the flow combined with Theorem \ref{thm:novolume} to derive:

\begin{Theorem}Suppose an Anosov thermostat given by the coupled vortex equations has a weak foliation of class $C^{2}$ and $m$ is odd. Then $A$ vanishes identically.
\label{thm:regweak}
\end{Theorem}

\begin{proof} When $m$ is odd there is an important additional symmetry in the flow: the flip $\sigma$ given by
$(x,v)\mapsto (x,-v)$. We note that this map is isotopic to the identity. If $\phi$ denotes the thermostat flow then,
$\sigma\circ\phi_{t}=\phi_{-t}\circ\sigma$. This relation easily implies that $\sigma$ maps the weak stable foliation to the unstable one. Hence, if one of them is of class $C^{2}$, the other one is also of class $C^2$.

As we have already mentioned, Theorem 4.6 in \cite{Ghy2} asserts that a smooth Anosov flow on a closed 3-manifold with weak stable and unstable foliations of class $C^{2}$,
is smoothly orbit equivalent to a quasi-fuchsian flow $\psi$ that depends on a pair of points $([g_1],[g_2])$
in Teichm\"uller space. The flow $\psi$ has smooth weak stable foliation $C^{\infty}$-conjugate to the
weak stable foliation of the constant curvature metric $g_1$ and smooth weak unstable foliation
$C^{\infty}$-conjugate to the weak unstable foliation of the constant curvature metric $g_2$. But since $\sigma$ is isotopic to the identity we must have
$[g_{1}]=[g_{2}]$ and $\psi$ is an ordinary geodesic flow preserving a volume form.
Thus our thermostat flow preserves a volume form and by Theorem \ref{thm:novolume} we must have $A=0$.
\end{proof}
\begin{Remark} It is instructive to discuss Theorem \ref{thm:regweak} in the light of the remarks in Section \ref{section:2-3} for $m=3$. As pointed out, in this case, the thermostat flow is a $C^{\infty}$ parametrisation of the geodesic foliation of a Hilbert metric. Benoist observes in \cite{BConvDivI} that the regularity of the weak foliations of the Hilbert geodesic flow coincides with the regularity of the boundary. Hence if the boundary of the strictly convex domain defining the Hilbert metric is $C^2$, then the associated thermostat flow also has $C^2$ weak foliations and therefore $A=0$. This implies that the convex domain is an ellipsoid, thus recovering a result of Benz\'ecri~\cite{Benz} for the case of 2-dimensional domains (note however, that the proof in \cite{Benz} is more direct and straightforward).
\end{Remark}
\section{The path geometry defined by a thermostat}

A thermostat naturally defines a path geometry and in this final section we show that the path geometry associated to the thermostat coming from a holomorphic differential $A$ of degree $m\geqslant 2$ is flat if and only if $A$ vanishes identically or $m=3$. The former case corresponds to the paths being the geodesics of a hyperbolic metric and the latter case to the paths being the geodesics of a convex projective structure. We first recall some elementary facts about path geometries while referring the reader to~\cite{BRY} for further details.  

An~\textit{(oriented) path geometry} on an oriented surface $M$ is given by an oriented line bundle $L$ on the projective circle bundle $\mathbb{S}M:=\left(TM\setminus\{0\}\right)/\R^+$ having the property that $L$ together with the vertical bundle of the projection map $\nu : \mathbb{S}M\to M$ spans the contact distribution of $\mathbb{S}M$. The~\textit{paths} of $L$ are the projections of its integral curves to $M$. Note that the orientation of $L$ naturally equips its paths with an orientation. 
\begin{Example}
Taking $M$ to be the oriented $2$-sphere $S^2$, we obtain a canonical path geometry $L_0$ whose paths are the great circles. In this case $\mathbb{S}S^2\simeq \mathrm{SO}(3)$ and $L_0$ is the line bundle defined by $\omega_2=\psi=0$, where we write the Maurer--Cartan form $\omega_{\mathrm{SO}(3)}$  of $\mathrm{SO}(3)$ as
$$
\omega_{\mathrm{SO}(3)}=\begin{pmatrix} 0 & -\omega_1 & -\omega_2 \\ \omega_1 & 0 & -\psi \\ \omega_2 & \psi & 0\end{pmatrix}
$$      
for left-invariant $1$-forms $\omega_1,\omega_2,\psi$ on $\mathrm{SO}(3)$. Moreover, we orient $S^2$ such that an orientation compatible volume form pulls back to $\mathrm{SO}(3)$ to become a positive multiple of $\omega_1\wedge\omega_2$ and orient $L_0$ in such a way that $\omega_1$ is positive on positive vectors of $L_0$. 
\end{Example}

\begin{Definition}
A path geometry $L$ on $M$ is called~\textit{flat}, if for every point $p \in M$, there exists a neighbourhood $U_p$ and an orientation preserving diffeomorphism $f : U_p \to V$ onto some open subset $V\subset S^2$, which maps the positively oriented paths contained in $U_p$ onto positively oriented great circles.
\end{Definition} 

Let now $F=X+\lambda V$ be a thermostat on the unit tangent bundle $SM$ of a oriented Riemannian $2$-manifold $(M,g)$. We henceforth identify $SM \simeq \mathbb{S}M$ in the obvious way. In doing so, we obtain a path geometry by defining $L:=\R F$ and by declaring vectors in $L$ to be positive if they are positive multiples of $F$.

Clearly, if a path geometry is flat, then it must have the property that its paths agree with the geodesics of some projective structure. In~\cite[Proposition 3.4]{MetPat2} it is shown that the path geometry defined by a thermostat $X+\lambda V$ shares its paths with the geodesics of some projective structure if and only if 
\begin{equation}\label{eq:curvcartgeom}
0=\frac{3}{2}\lambda+\frac{5}{3}VV\lambda+\frac{1}{6}VVVV\lambda.
\end{equation}
Using this fact we immediately obtain:
 
\begin{Theorem}
Let $(g,A)$ be a pair satisfying the coupled vortex equations $\bar{\partial}A=0$ and $K_g=-1+(m-1)|A|^{2}_{g}$. Then the path geometry defined by the thermostat associated to $(g,A)$ is flat if and only if $m=3$ or $A$ vanishes identically.
\end{Theorem}

\begin{proof}
Suppose the path geometry associated to $(g,A)$ is flat. Recall that for our choice $\lambda=a$ we have $VVa=-m^2 a$, hence~\eqref{eq:curvcartgeom} gives
$$
0=\left(\frac{1}{6}m^4-\frac{5}{3}m^2+\frac{3}{2}\right)a=\frac{1}{6}(m-1)(m+1)(m-3)(m+3)a.
$$
Consequently, $a$ and hence $A$ must vanish identically or $m=3$. 

Conversely, assume $A$ is a cubic differential satisfying $\ov{\partial} A=0$ and $K_g=-1+2|A|^2_g$. The path geometry associated to $(g,A)$ defines a properly convex projective structure on the oriented surface $M$. An oriented properly convex projective surface is an example of a surface carrying a $(G,X)$-structure where $X=\mathbb{S}^2$ is the oriented projective $2$-sphere and $G=\mathrm{SL}(3,\R$) its group of projective transformations, cf.~\cite{KimPap}. In particular, it follows that the path geometry associated to $(g,A)$ is flat.  
\end{proof}

\end{document}